\documentclass[a4paper, 12pt]{amsart}
\usepackage[margin=1in]{geometry}

\usepackage{amsmath}
\usepackage{amssymb}
\usepackage{calc}
\usepackage{enumitem}
\usepackage{graphicx}
\usepackage{tikz-cd}
\usepackage{url}
\usepackage{xcolor}
\usepackage{etoolbox}
\usepackage{tabularx}
\usepackage{hyperref}
\usepackage{mathtools}

\hypersetup{
  colorlinks   = true, 
  urlcolor     = {blue!50!black}, 
  linkcolor    = {blue!50!black}, 
  citecolor   = {red!50!black} 
}

\theoremstyle{plain}
\newtheorem{thm}{Theorem}[section]

\newtheorem{prop}[thm]{Proposition}
\newtheorem{lem}[thm]{Lemma}
\newtheorem*{thm*}{Theorem}
\newtheorem*{prop*}{Proposition}
\newtheorem*{lem*}{Lemma}
\newtheorem*{cor*}{Corollary}
\newtheorem*{claim*}{Claim}

\newtheorem{customthm}{Theorem}
\newenvironment{numthm}[1]
{\renewcommand\thecustomthm{#1}\customthm}
{\endcustomthm}

\newtheorem{customprop}{Proposition}
\newenvironment{numprop}[1]
{\renewcommand\thecustomprop{#1}\customprop}
{\endcustomprop}

\theoremstyle{definition}
\newtheorem{defn}[thm]{Definition}

\theoremstyle{remark}

\newtheorem*{term*}{Terminology}

\newtheorem*{qn*}{Question}

\newcommand{\alignsub}{\vphantom{-}}%

\DeclareMathOperator{\Age}{Age}
\DeclareMathOperator{\Aut}{Aut}
\DeclareMathOperator{\cl}{cl}

\DeclareMathOperator{\dplus}{d_{_{+}}\!}
\DeclareMathOperator{\dminus}{d_{_{--}}\!}
\DeclareMathOperator{\Edg}{E}
\DeclareMathOperator{\Emb}{Emb}

\DeclareMathOperator{\Nbhd}{N}
\DeclareMathOperator{\Nminus}{N_{_{\!--}}}
\DeclareMathOperator{\Nplus}{N_{_{\!+}}}
\DeclareMathOperator{\Or}{Or}

\renewcommand{\d}[1]{\ensuremath{\operatorname{d}\!{#1}}}

\newcommand{\N}{
	\mathbb{N}
}

\newcommand{\R}{
	\mathbb{R}
}

\newcommand{\mc}[1]{
	\mathcal{#1}
}

\newcommand{\Cgzero}{
	\mathcal{C}_{>0}
}

\newcommand{\CF}{
	\mathcal{C}_F
}

\newcommand{\ov}[1]{
	\overline{#1}
}

\newcommand{\ex}{
	\exists\,
}
\newcommand{\sub}{
	\subseteq
}

\newcommand{\fin}{
	\subseteq_{\text{fin.\!}}
}

\newcommand{\Honza}{Hubi\v{c}ka }

\newcommand{\Jarik}{Ne\v{s}et\v{r}il }
\newcommand{\Jarikn}{Ne\v{s}et\v{r}il}
\newcommand{\Fr}{Fra\"{i}ss\'{e} }
\newcommand{\Frn}{Fra\"{i}ss\'{e}}
\newcommand{\KPT}{Kechris-Pestov-Todor\v{c}evi\'{c} }

\expandafter\patchcmd\csname\string\proof\endcsname{\normalparindent}{0pt }{}{}

\title{Sparse graphs and the fixed points on type spaces property}
\author{Rob Sullivan}
\address{Rob Sullivan, Institut f\"{u}r Mathematische Logik, Universit\"{a}t  M\"{u}nster,  
Einsteinstraße 62,
48149  M\"{u}nster,
Germany}
\email{r.sullivan@uni-muenster.de}

\thanks{This project formed the first part of the PhD of the author at Imperial College London, under the supervision of Prof.\ David Evans.}

\date{\today}



\begin{document}

\begin{abstract}
    We examine the topological dynamics of the automorphism groups of $\omega$-categorical sparse graphs resulting from Hrushovski constructions. Specifically, we consider the fixed points on type spaces property, which a structure $M$ has if, for all $n \in \N$, every $\Aut(M)$-subflow of the space $S_n(M)$ of $n$-types has a fixed point. Extending a result of Evans, \Honza and \Jarikn, we show that there exists an $\omega$-categorical structure $M$, resulting from a Hrushovski construction, such that no $\omega$-categorical expansion of $M$ has the fixed points on type spaces property.
\end{abstract}

\maketitle

\section{Introduction}
 
    The paper \cite{EHN19} is concerned with the topological dynamics of the automorphism groups of sparse graphs, in the context of the \KPT correspondence (\cite{KPT05}). One of the key results of \cite{EHN19} is the following:

    \begin{thm*}[{\cite[Theorem 1.2]{EHN19}}]
        There exists an $\omega$-categorical structure $M$ such that no $\omega$-categorical expansion has an extremely amenable automorphism group.
    \end{thm*}

    We recall that, for a Hausdorff topological group $G$, a \emph{$G$-flow} is a continuous action of $G$ on a nonempty compact Hausdorff space $X$, and we say that $G$ is \emph{extremely amenable} if every $G$-flow has a $G$-fixed point.

    In this paper, we show that the above result holds even in the context of a more restricted class of flows: subflows of type spaces. Let $M$ be a relational structure. Following \cite{MS23}, we say that $M$ has the \emph{fixed points on type spaces property} (FPT), if, for each $n \in \N_+$, every subflow of $S_n(M)$ has an $\Aut(M)$-fixed point, where $S_n(M)$ denotes the Stone space of $n$-types with parameters from $M$ and the action is given by translation of parameters in formulae. This property is studied in depth in \cite{MS23}, and may be thought of as a restriction of extreme amenability to a subclass of flows which occur naturally in a model-theoretic context. 
    
    The main result of this paper is as follows.

    \begin{numthm}{\ref{FPTMF}}
        There is an $\omega$-categorical structure $M$ such that no $\omega$-categorical expansion has FPT, the fixed points on type spaces property.
    \end{numthm}

    The structure $M$ appearing in both these results is a particular type of $\omega$-categorical sparse graph known as an $\omega$-categorical Hrushovski construction (first introduced in \cite{Hru88} -- a clear introductory exposition may be found in \cite{Eva13}). A graph $A$ is \emph{$k$-sparse} if for all finite $B \sub A$, the number of edges of $B$ is at most $k$ times the number of vertices of $B$.

    The proof strategy for Theorem \ref{FPTMF} is as follows. A central fact in the analysis of sparse graphs is that a graph is $k$-sparse iff it is $k$-orientable: its edges may be directed so that each vertex has at most $k$ out-edges. This fact is well known to graph theorists (\cite{Nas64}), and the proof is by Hall's Marriage Theorem (see Prop.\ \ref{oriffsparse}).
    
    For any $k$-sparse graph $M$, the space $\Or(M)$ of $k$-orientations of $M$ (with the subspace topology from $2^{M^2}$) gives an $\Aut(M)$-flow (Lemma \ref{OrMisflow}). As in \cite{EHN19}, we specialise to the case $k = 2$ (results generalise straightforwardly to any $k$). Theorem 1.2 of \cite{EHN19}, the result of Evans, \Honza and \Jarik mentioned above, then immediately results from the following, using the Ryll-Nardzewski theorem:

    \begin{numprop}{\ref{MFfixorinforbs}}[adapted from {\cite[Theorem 3.7]{EHN19}}]
        Let $M$ be an infinite $2$-sparse graph in which all vertices have infinite degree. Let $G = \Aut(M)$.
		
        Consider the $G$-flow $G \curvearrowright \Or(M)$. If $H \leq G$ fixes a $2$-orientation of $M$, then $H$ has infinitely many orbits on $M^2$.
    \end{numprop}

    To prove Theorem \ref{FPTMF}, we also use the above result. Letting $M$ be the $\omega$-categorical Hrushovski construction detailed in Section \ref{omegacatsparsesec}, we define a notion of when a $1$-type \emph{encodes} an orientation of $M$. We then define an $\Aut(M)$-flow morphism $u : S_1(M) \to 2^{{M}^2}$ which sends each orientation-encoding $1$-type to the orientation it encodes. Let $M'$ be an expansion of $M$ with FPT, and let $H = \Aut(M')$. Then $H$ must fix a point in the subflow of orientation-encoding $1$-types, so fixes an orientation. We then use Proposition \ref{MFfixorinforbs} to see that $H$ has infinitely many orbits on $M^2$, so $H$ is not oligmorphic, and therefore by the Ryll-Nardzewski theorem we see that $M'$ is not $\omega$-categorical. Thus $M$ has no $\omega$-categorical expansion with FPT.

\subsection*{Acknowledgements} The author would like to thank David Evans for his supervision during this project, which formed the first half of the author's PhD thesis. The author would also like to thank the anonymous referee for their helpful comments, including a simplification of the proof of the main theorem.

\section{Background} \label{backgroundchap}

    In this section, we present the sufficient background material on topological dynamics, sparse graphs and \Fr classes with distinguished substructures (``strong \Fr classes") in order to be able to construct the $\omega$-categorical examples of sparse graphs ($\omega$-categorical Hrushovski constructions) given in Section \ref{omegacatsparsesec}.

    We assume that the reader is familiar with the classical \Fr theory, the pointwise convergence topology on automorphism groups of first-order structures and the Ryll-Nardzewski theorem. (The background for these three topics can be found in \cite[Chapter 7]{Hod93} and \cite[Sections 1-2]{Eva13}.)
	
    The background material in this section has been mostly adapted from \cite{EHN19} and \cite{Eva13}.
    
    All first-order languages considered in this article will be countable and relational.

\subsection{Topological dynamics} \label{top dyn intro}

    A central object of study in topological dynamics is the following (see \cite{Aus88} for a thorough background):
	
    \begin{defn}
        A \emph{$G$-flow} is a continuous action $G \curvearrowright X$ of a Hausdorff topological group $G$ on a nonempty compact Hausdorff topological space $X$.
    \end{defn}
	
    We will often simply write $X$ to refer to the $G$-flow $G \curvearrowright X$ when this is clear from context. Given a $G$-flow on $X$, $\ov{G \cdot x}$, the orbit closure of a point $x \in X$, is a $G$-invariant compact subset of $X$. In general, a nonempty compact $G$-invariant subset $Y \subseteq X$ defines a \emph{subflow} by restricting the $G$-action to $Y$.
	
    Let $X, Y$ be $G$-flows. A \emph{$G$-flow morphism} $X \to Y$ is a continuous map $\alpha: X \to Y$ such that $\alpha(g \cdot x) = g \cdot \alpha(x)$ (this property is called \emph{$G$-equivariance}). A surjective $G$-flow morphism $X \to Y$ is called a \emph{factor} of $X$, and we will also say that $Y$ is a factor of $X$ when the morphism is contextually implied. Bijective $G$-flow morphisms are isomorphisms, as they are between compact Hausdorff spaces.

\subsection{Graphs}

    We work with graphs in first-order logic as follows. Let $\mc{L}$ be a first-order language consisting of a single binary relation symbol $E$. A \emph{graph} consists of an $\mc{L}$-structure $(A, E^A)$ where the binary relation $E^A \sub A^2$ is symmetric and irreflexive. We call $A$ the \emph{vertex set}, and write $E_A$ for the set of unordered pairs $\{a, b\}$ such that $(a, b) \in E^A$. We call $E_A$ the \emph{edge set}, and this will usually be the relevant set we work with in this paper, following the usual graph-theoretic definition of a graph -- rather than the symmetric set $E^A$ of ordered pairs, which we only introduce for the sake of first-order structure formalism. We will usually just write $A$ to denote the graph $(A, E^A)$ when this is clear from context. We will often write $\sim$ instead of $E$ in formulae to indicate adjacency. 
    
    By the above definition, here we only work with \emph{simple} graphs: graphs having no loops on a single vertex or multiple edges between two vertices.
	
    \begin{defn}
        Let $(A, E^A)$ be a graph. A set $\rho^A \sub A^2$ is an \emph{orientation} of $(A, E^A)$ if:
        \begin{itemize}
            \item $\rho^A \sub E^A$;
            \item for each $(x, y) \in E^A$, exactly one of $(x, y), (y, x)$ is in $\rho^A$.
        \end{itemize}
		
        We may visualise the above definition as follows: an orientation of a graph consists of a direction for each edge.
        
        Note that the above definition implies that $\rho^A$ contains no directed loops or directed $2$-cycles. We will refer to $(A, E^A, \rho^A)$ as an \emph{oriented graph}.
    \end{defn}
    \begin{defn} \label{basicordefs}
        Let $(A, E^A, \rho^A)$ be an oriented graph.
		
        If $(x, y) \in \rho^A$, we refer to $(x, y)$ as an \emph{out-edge} of $x$ and as an \emph{in-edge} of $y$. We call $y$ an \emph{out-vertex} of $x$, and $x$ an \emph{in-vertex} of $y$.
		
        The \emph{out-neighbourhood} $\Nplus(x)$ of $x$ consists of the out-vertices of $x$. The \emph{in-neighbourhood} $\Nminus(x)$ of $x$ consists of the in-vertices of $x$. The \emph{out-degree} $\dplus(x)$ of $x$ is defined to be $\dplus(x) = |\Nplus(x)|$, and the \emph{in-degree} $\dminus(x)$ of $x$ is defined to be $\dminus(x) = |\Nminus(x)|$.		
    \end{defn}
  
    When we refer to a subgraph of a graph, or an oriented subgraph of an oriented graph, we mean a substructure in the model-theoretic sense. For graph theorists, these substructures would usually be referred to as \emph{induced} subgraphs.
	
    (We use the full notation for structures in this section for clarity, but henceforth we will usually denote graphs $(A, E^A)$ by $A$, and oriented graphs $(A, E^A, \rho^A)$ by $A$ or $(A, \rho^A)$.) 

\subsection{Sparse graphs}

    \begin{defn}
        Let $k \in \N_+$. A graph $A$ is \emph{$k$-sparse} if for all $B \fin A$, we have $|E_B| \leq k |B|$.
    \end{defn}
	
    \begin{defn}
        Let $(A, \rho^A)$ be an oriented graph. Let $k \in \N_+$. We call $\rho^A$ a \emph{$k$-orientation} if for $x \in A$, we have $\dplus(x) \leq k$. We refer to $(A, \rho^A)$ as a \emph{$k$-oriented graph}.
        
        If an undirected graph $A$ has a $k$-orientation, we say it is \emph{$k$-orientable}.
    \end{defn}
	
    The following proposition is well-known to graph theorists (\cite{Nas64}), and will be a key tool here. We present the proof as it is relatively brief.
	
    \begin{prop}[{\cite[Theorem 3.4]{EHN19}}] \label{oriffsparse}
        Let $A$ be a countable graph. Then $A$ is $k$-orientable iff it is $k$-sparse.
    \end{prop}
    \begin{proof}
        $\Rightarrow:$ straightforward. $\Leftarrow:$ We prove the statement for finite $A$, and then the statement for countably infinite $A$ follows by a straightforward K\H{o}nig's lemma argument. We wish to produce a $k$-orientation of $A$, and to do this we must direct each edge. We will use Hall's Marriage Theorem (\cite[III.3]{Bol98}), which for the convenience of the reader we briefly restate: for a finite bipartite graph $B$ with left set $X$ and right set $Y$, there is an $X$-saturated matching iff $|W| \leq |\Nbhd_B(W)|$ for $W \sub X$. (Here $\Nbhd_B(W)$ denotes the neighbourhood of $W$ in $B$.)
		
        Form a bipartite graph $B$ with left set $E_A$ and right set $A \times [k]$, and place an edge between $e \in E_A$ and $(x, i) \in A \times [k]$ if $x \in e$. Given a left-saturated matching, if $e$ is matched to $(x, i)$, we orient $e$ outwards from $x$, and this gives a $k$-orientation of $A$.
		
        To see that a left-saturated matching exists, take $W \sub E_A$. Let $V$ be the set of vertices of the edges which lie in $W$. Then $|\Nbhd_B(W)| = k|V|$, and as $A$ is $k$-sparse, we have that $k|V| \geq |\Edg_A(V)|$, where $\Edg_A(V)$ is the set of edges in $A$ whose vertices lie in $V$. As $|\Edg_A(V)| \geq |W|$, by Hall's Marriage Theorem there exists a left-saturated matching of the bipartite graph $B$.
	\end{proof}
	
    For presentational simplicity, we will work with $k = 2$. Our results generalise straightforwardly for $k > 2$.
	
    \textbf{Note:} in this paper, we may occasionally say ``oriented graph" to in fact mean ``$2$-oriented graph". We will try to avoid this in general, but when this does occur the meaning will be clear from context.

    \begin{defn}
        Let $M$ be a $2$-sparse graph. We let $\Or(M) \sub 2^{M^2}$ denote the topological space of $2$-orientations of $M$, where the topology is given by the subspace topology from the Cantor space $2^{M^2}$. 
    \end{defn}
    \begin{lem} \label{OrMisflow}
        Let $M$ be a $2$-sparse graph. Then $\Or(M)$ is an $\Aut(M)$-flow with the natural action \[g \cdot \rho = \{(gx, gy) : (x, y) \in \rho\}.\]
    \end{lem}
    \begin{proof}
        By Prop.\ \ref{oriffsparse}, we see that $\Or(M)$ is non-empty, and it is immediate that $\Or(M)$ is $\Aut(M)$-invariant. It therefore remains to show that $\Or(M)$ is closed in $2^{M^2}$: if $\sigma \in 2^{M^2}$ is not a $2$-orientation, then this is witnessed on a finite set, so $2^{M^2} \setminus \Or(M)$ is open.
    \end{proof}

\subsection{Graph predimension}

    One way to characterise $2$-sparsity is in terms of a particular notion of \emph{graph predimension}.
	
    \begin{defn}
        Let $A$ be a finite graph. We define the \emph{predimension} $\delta(A)$ of $A$ to be $\delta(A) = 2|A| - |E_A|$.
		
        For $B \sub A$, we define the \emph{relative predimension of $A$ over $B$} to be $\delta(A/B) = \delta(A) - \delta(B)$.
    \end{defn}

    We immediately see that, for $A$ a finite graph, $A$ is $2$-sparse iff for all $B \sub A$ we have $\delta(B) \geq 0$.

\subsection{Strong classes} \label{strongclassessec}

    For the $\omega$-categorical Hrushovski constructions in Section \ref{omegacatsparsesec}, we will need to take a class of sparse graphs where we only consider particular distinguished embeddings between structures in the class, and for this we require the definition below. In the subsequent section, we will construct \Fr classes where we only permit these distinguished embeddings between finite structures.

    \begin{defn} \label{strongclassdef}
        Let $\mc{K}$ be a class of finite $\mc{L}$-structures closed under isomorphisms. Let $\mc{S} \sub \Emb(\mc{K})$ be a class of embeddings between structures in $\mc{K}$ satisfying the following:
        \begin{enumerate}
            \item[(S1)] $\mc{S}$ contains all isomorphisms;
            \item[(S2)] $\mc{S}$ is closed under composition;
            \item[(S3)] if $f : A \to C$ is in $\mc{S}$ and $f(A) \sub B \sub C$ with $B \in \mc{K}$, then $f : A \to B$ is in $\mc{S}$.
        \end{enumerate}
        Then we call $(\mc{K}, \mc{S})$ a \emph{strong class}, and call the elements of $\mc{S}$ \emph{strong embeddings}. 
		
        (This is originally due to Hrushovski - see \cite{Hru88}. An accessible exposition of strong classes is in \cite[Section 3]{Eva13}.)
		
        If $A, B \in \mc{K}$, $A \sub B$ and the inclusion map $\iota : A \hookrightarrow B$ is in $\mc{S}$, then we write $A \leq B$ and say $A$ is a \emph{strong substructure} of $B$. We then have that:
		
        \begin{enumerate}
            \item[(L1)] $\leq$ is reflexive;
            \item[(L2)] $\leq$ is transitive;
            \item[(L3)] if $A \leq C$ and $A \sub B \sub C$ with $B \in \mc{K}$, then $A \leq B$.
        \end{enumerate}
		
        We will often write $(\mc{K}, \leq)$ instead of $(\mc{K}, \mc{S})$, and we will refer to the elements of $\mc{S}$ as \emph{$\leq$-embeddings}.
    \end{defn}
	
    If $(\mc{K}, \leq)$ is a strong class (i.e.\ $\mc{S}$ satisfies (S1), (S2), (S3)), then we have that for $f : A \to B$ in $\mc{S}$, if $X \leq A$, then $f(X) \leq B$.

    \begin{defn} \label{infstrong}
        Suppose $(\mc{K}, \leq)$ is a strong class. Let $A_0 \leq A_1 \leq \cdots$ be an increasing $\leq$-chain of structures in $\mc{K}$, and let $M = \bigcup_{i \in \N} A_i$. Let $A \fin M$. 
        
        Then we write $A \leq M$, and say that $A$ is a \emph{strong substructure} of $M$, or that $A$ is \emph{$\leq$-closed} in $M$, to mean that there is some $A_i$ ($i \in \N$) with $A \leq A_i$.

        Given $A \in \mc{K}$ and an embedding $f : A \to M$, we will likewise say that $f$ is a $\leq$-embedding if $f(A) \leq M$.
    \end{defn}
    
    The above definition is independent of the choice of $\leq$-chain. To see this, suppose $M$ is also the union of the elements of an increasing $\leq$-chain $B_0 \leq B_1 \leq \cdots$ of $\mc{K}$-structures. Take any $A_i$ ($i \in \N$). Then $A_i \sub B_j$ for some $j \in \N$, and $B_j \sub A_k$ for some $k \geq i$. As $A_i \leq A_k$, by (L3) we have $A_i \leq B_j$.
	
    Let $g \in \Aut(M)$. Take a pair $A_i \leq A_j$ ($i < j$). Then $g|_{A_j} : A_j \to gA_j$ is an isomorphism, so $g|_{A_j} \in \mc{S}$, and so $gA_i \leq gA_j$. Thus $M$ is also the union of the increasing $\leq$-chain $gA_0 \leq gA_1 \leq \cdots$. So if $A \leq M$, then $gA \leq M$: that is, all $g \in \Aut(M)$ preserve $\leq$.

\subsection{\Fr theory for strong classes} \label{strongFrsec}
	
    We now develop an analogue of the classical \Fr theory for strong classes. We omit the proofs and state the relevant material as a series of definitions and lemmas. (For the classical \Fr theory, originally developed in \cite{Fra54}, see \cite[Chapter 7]{Hod93}, and for a more complete treatment of \Fr theory for strong classes, see \cite[Section 3]{Eva13}.)
    	
    \begin{defn}
        Let $(\mc{K}, \leq)$ be a strong class of $\mc{L}$-structures.
        \begin{itemize}
            \item $(\mc{K}, \leq)$ has the \emph{joint embedding property} (JEP) if for $A_0, A_1 \in \mc{K}$, there is $B \in \mc{K}$ with $\leq$-embeddings $f_0 : A_0 \to B, f_1 : A_1 \to B$.
            \item $(\mc{K}, \leq)$ has the \emph{amalgamation property} (AP) if, for any pair of $\leq$-embeddings $B_0 \xleftarrow{f_0} A \xrightarrow{f_1} B_1$, there exists $C \in \mc{K}$ and a pair of $\leq$-embeddings $B_0 \xrightarrow{g_0} C \xleftarrow{g_1} B_1$ such that $g_0 \circ f_0 = g_1 \circ f_1$.
            \item For $A, B_0, B_1 \in \mc{K}$ with $A \leq B_0, B_1$, the \emph{free amalgam} $C$ of $B_0, B_1$ over $A$ is the $\mc{L}$-structure $C$ whose domain is the disjoint union of $B_0, B_1$ over $A$ and whose relations $R^C$ are exactly the unions $R^{B_0} \cup R^{B_1}$ of the relations $R^{B_0}, R^{B_1}$ on $B_0, B_1$ (for $R$ a relation symbol in $\mc{L}$). If for all $\mc{L}$-structures $A, B_0, B_1 \in \mc{K}$ with $A \leq B_0, B_1$ we have that the free amalgam $C$ of $B_0, B_1$ over $A$ is in $\mc{K}$ with $B_0, B_1 \leq C$, then we say that $(\mc{K}, \leq)$ is a \emph{free amalgamation class}.
        \end{itemize}
		
        We will usually not mention the distinguished class of embeddings in our terminology, as it will be clear from context and the fact that we are working with strong classes. For instance, we say that $(\mc{K}, \leq)$ has the amalgamation property, even though perhaps more strictly we should say that $(\mc{K}, \leq)$ has the $\leq$-amalgamation property. 
    \end{defn}
	
    In the following definitions and lemmas, let $(\mc{K}, \leq)$ be a strong class, and let $M$ be the union of an increasing $\leq$-chain $A_1 \leq A_2 \leq \cdots$ of finite structures in $(\mc{K}, \leq)$.
    
    \begin{defn}
        The $\leq$-\emph{age} of $M$, written $\Age_\leq(M)$, is the class of $A \in \mc{K}$ such that there is a $\leq$-embedding $A \to M$.
    \end{defn}
	
    The class $(\Age_\leq(M), \leq)$ is a $\leq$-hereditary strong subclass of $(\mc{K}, \leq)$, and it has the $\leq$-joint embedding property.
	
    \begin{defn}
        $M$ has the $\leq$-\emph{extension property} if for all $A, B \in \Age_\leq(M)$ and $\leq$-embeddings $f : A \to M, g : A \to B$, there exists a $\leq$-embedding $h : B \to M$ with $h \circ g = f$.
		
        $M$ is $\leq$-\emph{ultrahomogeneous} if each isomorphism $f : A \to A'$ between strong substructures $A, A'$ of $M$ extends to an automorphism of $M$.

        (Again, when it is clear from context, we will often omit the $\leq$- prefix and just say that $M$ has the extension property or is ultrahomogeneous.) 
    \end{defn}
	
    \begin{lem}
        Let $M'$ also be a union of an increasing $\leq$-chain in $\mc{K}$. Suppose $M, M'$ have the same $\leq$-age and both have the $\leq$-extension property. Then $M, M'$ are isomorphic.
    \end{lem}
	
    \begin{lem}
        $M$ is $\leq$-ultrahomogeneous iff $M$ has the $\leq$-extension property.
    \end{lem}
    \begin{lem}
        Suppose $M$ is $\leq$-ultrahomogeneous. Then the class $(\Age_\leq(M), \leq)$ has the amalgamation property.
    \end{lem}
	
    \begin{defn}
        Let $(\mc{K}, \leq)$ be a strong class. We say that $(\mc{K}, \leq)$ is an \emph{amalgamation class} or \emph{\Fr class} if $(\mc{K}\, \leq)$ contains countably many isomorphism types, contains structures of arbitrarily large finite size, and has the joint embedding and amalgamation properties.
    \end{defn}
	
    \begin{thm}[\Frn-Hrushovski]
        Let $(\mc{K}, \leq)$ be an amalgamation class. Then there is a structure $M$ which is a union of an increasing $\leq$-chain in $\mc{K}$ such that $M$ is $\leq$-ultrahomogeneous and $\Age_\leq(M) = \mc{K}$, and $M$ is unique up to isomorphism amongst structures with these properties.
		
        We call this structure the \emph{\Fr limit} or \emph{generic structure} of $\mc{K}$.
    \end{thm}

\section{\texorpdfstring{$\omega$}{omega}-categorical sparse graphs} \label{omegacatsparsesec}

    The material in this section is based on \cite{EHN19} and the unpublished notes \cite{Eva13}, with some minor modifications, and constitutes further background required for Section \ref{resultsec}.

    We now construct an amalgamation class of sparse graphs whose \Fr limit is $\omega$-categorical. Specifically, this will be a version of the $\omega$-categorical Hrushovski construction $M_F$, first seen in \cite{Hru88}. We will do this by defining a notion of closure (i.e.\ a particular notion of strong substructure), $d$-closure, which will be uniformly bounded. The relevance of this can be seen in the lemma below.
	
    \begin{lem}[{\cite[Remark 2.8]{EHN19}}] \label{unifbdomegacat}
        Let $(\mc{K}, \leq)$ be an amalgamation class such that for each $n \in \N$, $(\mc{K}, \leq)$ has only finitely many isomorphism classes of structures of size $n$. Suppose there is a function $h : \N \to \N$ such that for $B \in \mc{K}$ and $A \sub B$ with $|A| \leq n$, there exists $A \sub C \leq B$ with $|C| \leq h(n)$. 
        
        Then the \Fr limit $M$ of $(\mc{K}, \leq)$ is $\omega$-categorical.
    \end{lem}
    (The function $h$ will be a uniform bound on the size of $\leq$-closures.)
    \begin{proof}
        By the Ryll-Nardzewski theorem, it suffices to show that, for $n \geq 1$, $\Aut(M)$ has finitely many orbits on $M^n$. Take $n \geq 1$. As there are only finitely many isomorphism types of structures of size $\leq h(n)$ in $\mc{K}$ and $M$ is $\leq$-ultrahomogeneous, we have that $\Aut(M)$ has finitely many orbits on $\{\ov{c} \in M^{h(n)} : \ov{c} \leq M\}$. We can extend any $\ov{a} \in M^n$ to an element of this set (note that in ordered tuples, we can have repeats of elements). If $\ov{a}, \ov{a}'$ are not in the same orbit, then nor will their extensions be, so we are done.
    \end{proof}
	
    \begin{defn}
        Let $\Cgzero$ be the class of finite graphs $A$ such that for nonempty $B \sub A$, we have $\delta(B) > 0$.
    \end{defn}

    We note that for $A \in \Cgzero$, if $A' \sub A$ then $A' \in \Cgzero$.
 
    \begin{defn}
        Take $A, B \in \Cgzero$ with $A \sub B$. We say that $A$ is $d$-closed in $B$, written $A \leq_d B$, if for all $A \subsetneq C \sub B$, we have $\delta(A) < \delta(C)$.
    \end{defn}

    \begin{lem}[Submodularity. {\cite[Lemma 3.7]{Eva13}}] \label{submodlemma}
        Let $A$ be a finite graph, and let $B, C \sub A$. Then $\delta(B \cup C) \leq \delta(B) + \delta(C) - \delta(B \cap C)$. We have equality iff $E_{B \cup C} = E_B \cup E_C$, i.e.\ $B, C$ are freely amalgamated over $B \cap C$ in $A$.
    \end{lem}
    
    The proof of the above lemma is straightforward. We now prove some basic properties of $\leq_d$.
    
    \begin{lem}[{\cite[Lemma 3.10]{Eva13}}] \label{basic leq_d properties}
        Let $B \in \Cgzero$.
        \begin{enumerate}
            \item $A \leq_d B$, $X \sub B$ $\Rightarrow$ $A \cap X \leq_d X$.
            \item $A \leq_d C \leq_d B$ $\Rightarrow$ $A \leq_d B$.
            \item $A_1, A_2 \leq_d B$ $\Rightarrow$ $A_1 \cap A_2 \leq_d B$.
        \end{enumerate}
    \end{lem}

    \begin{proof} \hfill 
        \begin{enumerate}
            \item Take $A \cap X \subsetneq Y \sub X$. Note that $A \cap Y = A \cap X$. By submodularity,
            \begin{align*}
                \delta(A \cup Y) &\leq \delta(A) + \delta(Y) - \delta(A \cap Y)\\ &= \delta(A) + \delta(Y) - \delta(A \cap X),
            \end{align*}
            so $\delta(Y) - \delta(A \cap X) \geq \delta(A \cup Y) - \delta(A) > 0$, using the fact that $A \subsetneq A \cup Y \sub B$.
			
            \item We may assume $A \neq C$. Take $A \subsetneq X \sub B$. By (1) applied to $C \leq_d B$ and $X \sub B$, we have $C \cap X \leq_d X$. Also we have $A \sub C \cap X \sub C$. So, as $A \leq_d C$, we have $\delta(A) < \delta(X)$.

            \item By (1), $A_1 \cap A_2 \leq_d A_1$. Then use (2).
	\end{enumerate}
    \end{proof}

    For $B \in \Cgzero$, by part (3) of the previous lemma we see that for $A \sub B$ we have that $\bigcap_{A \sub A' \leq_d B} A' \leq_d B$, so we can define the \emph{$d$-closure} of $A$ in $B$ as this intersection, written $\cl_B^d(A)$.
	
    \begin{lem}[{\cite[Lemma 3.12]{Eva13}}] \label{predimclosure}
        Let $B \in \Cgzero$ and let $A \sub B$. Then $\delta(A) \geq \delta(\cl_B^d(A))$.
    \end{lem}
    \begin{proof}
        Amongst all $A \sub X \sub B$, consider those for which $\delta(X)$ is smallest, and then out of these choose a $C$ of greatest size. By the first stage of selection, we have $\delta(C) \leq \delta(A)$, and by the second stage, if $C \subsetneq D \sub B$ then $\delta(C) < \delta(D)$, so $C \leq_d B$. So $\cl_B^d(A) \sub C \sub B$, and as $\cl_B^d(A) \leq_d B$, we have $\delta(\cl_B^d(A)) \leq \delta(C)$.
	\end{proof}
 
    \begin{lem}[{\cite[Lemma 3.15]{Eva13}}] \label{Cgzfreeamalg}
        $(\mc{C}_{>0}, \leq_d)$ is a free amalgamation class.
    \end{lem}
    \begin{proof}
        It only remains to check the free amalgamation property (which implies JEP). We prove a stronger claim. Given $A, B_1, B_2 \in \Cgzero$ such that $A \leq_d B_1$ and $A \subseteq B_2$, with $B_1, B_2 \subseteq E$, where $E$ is the free amalgam of $B_1, B_2$ over $A$, we claim that $B_2 \leq_d E$. Once we have the claim, note that $\varnothing \leq_d B_2 \leq_d E$ implies that $E \in \mc{C}_{>0}$.

        Take $B_2 \subsetneq X \subseteq E$. Then letting $Y = X \cap B_1$, we have $Y \supsetneq A$ and $X = B_2 \cup Y$, and $X$ is the free amalgam of $B_2, Y$ over $A$. So \[\delta(X) = \delta(B_2 \cup Y) = \delta(B_2) + \delta(Y) - \delta(A),\] and so, as $A \leq_d B_1$, \[\delta(X) - \delta(B_2) = \delta(Y) - \delta(A) > 0.\]
    \end{proof}
	
    The \Fr limit $M_{>0}$ of $(\Cgzero, \leq_d)$ is not $\omega$-categorical, as for $A \fin M_{>0}$, there is no uniform bound on $\left|\cl^d(A)\right|$ in terms of $|A|$.
	
    To construct $\omega$-categorical examples, as mentioned at the start of this section, we consider subclasses of $\mc{C}_{>0}$ in which $d$-closure is uniformly bounded.
	
    \begin{defn} \label{CFdef}
        Let $F : \R_{\geq 0} \to \R_{\geq 0}$ be a continuous, strictly increasing function with $F(0) = 0$ and $F(x) \to \infty$ as $x \to \infty$. We define
        \begin{center}
            $\mc{C}_F := \{B \in \mc{C}_{>0} : \delta(A) \geq F(|A|)$ \text{for all} $A \subseteq B\}$.
        \end{center}
    \end{defn}

    Note that if $B \in \CF$ and $C \sub B$, then $C \in \CF$.
 
    \begin{lem}[{\cite[Theorem 3.19]{Eva13}, \cite[Theorem 4.14]{EHN19}}] \hfill
        \begin{enumerate}
            \item For $B \in \mc{C}_F$, $A \subseteq B$, we have $|\cl^d_B(A)| \leq F^{-1}(2|A|)$.
            \item If $(\mc{C}_F, \leq_d)$ is an amalgamation class, then its \Fr limit $M_F$ is $\omega$-categorical.
        \end{enumerate}
    \end{lem}
	
    \begin{proof} \hfill
        \begin{enumerate}
            \item From Lemma \ref{predimclosure}, as $\cl^d_B(A) \in \mc{C}_F$, we have $F(|\cl^d_B(A)|) \leq \delta(\cl^d_B(A)) \leq \delta(A) \leq 2|A|$.
            \item This follows from Lemma \ref{unifbdomegacat}.
        \end{enumerate}
    \end{proof}
	
    \begin{defn} \label{infleqd}
        Suppose that $(\CF, \leq_d)$ is an amalgamation class, and write $M_F$ for its \Fr limit. 
        
        For $A \sub M_F$ with $A$ infinite, we say that $A \leq_d M_F$ if $A \cap X \leq_d X$ for all finite $X \sub M_F$.
        
        (Note that if $A$ is finite, then $A \leq_d M_F$ iff $A \cap X \leq_d X$ for all finite $X \sub M_F$, by part (1) of Lemma \ref{basic leq_d properties}, so this definition is consistent.)
        
        Similarly we define $\cl^d_{M_F}(A)$ as the smallest $\leq_d$-closed subset of $M_F$ containing $A$. (This is well-defined: intersections of $\leq_d$-closed subsets of $M_F$ are $\leq_d$-closed, by part (2) of Lemma \ref{basic leq_d properties}.)

        Let $A$ be a graph, possibly infinite, which is embeddable in $M_F$. We say that an embedding $f : A \to M_F$ is a \emph{$\leq_d$-embedding} if $f(A) \leq_d M_F$.
    \end{defn}

    We now describe a method for constructing the control function $F$ to ensure that $(\CF, \leq_d)$ is a free amalgamation class.
 
    \begin{lem}[adapted from {\cite[Example 3.20]{Eva13}, \cite[Example 4.15]{EHN19}}] \label{CFfreeamalg}
        Let $n \in \N$. Let $F$ be as in Definition \ref{CFdef}, and assume additionally that:
        \begin{itemize}
            \item $F$ is piecewise smooth;
            \item its right derivative $F'$ is decreasing;
            \item $F'(x) \leq 1/x$ for $x > n$;
            \item for $A, B_1, B_2 \in \mc{C}_F$ with $A \leq_d B_1, B_2$ and $|B_1| < n, |B_2| < n$, the free amalgam of $B_1, B_2$ over $A$ lies in $\CF$.
        \end{itemize} 
        
        Then $(\mc{C}_F, \leq_d)$ is a free amalgamation class.
    \end{lem}
    \begin{proof}
        Let $A, B_1, B_2 \in \mc{C}_F$, with $A \leq_d B_1, B_2$. We may assume $|B_1| \geq n$ and $|B_1| \geq |B_2|$. Let $E$ be the free amalgam of $B_1, B_2$ over $A$. By Lemma \ref{Cgzfreeamalg}, $E \in \mc{C}_{>0}$ and $B_1, B_2 \leq_d E$. We need to show that $E \in \mc{C}_F$. Assuming $E \neq B_1, B_2$, we have $A \neq B_1, B_2$. Suppose $X \subseteq E$: we need to show that $\delta(X) \geq F(|X|)$. As $X$ is the free amalgam of $B_1 \cap X$, $B_2 \cap X$ over $A \cap X$ and as $A \cap X \leq_d B_i \cap X$, it suffices to check just for $X = E$.
		
        We have that \[\delta(E) = \delta(B_1) + \delta(B_2) - \delta(A) = \delta(B_1) + (|B_2| - |A|)\frac{\delta(B_2) - \delta(A)}{|B_2| - |A|}.\] As $|B_1| \geq |B_2|$ and as $A \leq_d B_1$ with $A \neq B_1$, we have \[\frac{\delta(B_2) - \delta(A)}{|B_2| - |A|} \geq \frac{1}{|B_1|}.\] So \[\delta(E) \geq \delta(B_1) + \frac{|B_2| - |A|}{|B_1|} \geq F(|B_1|) + \frac{|B_2| - |A|}{|B_1|},\] and as the conditions on $F$ ensure that $F(x + y) \leq F(x) + y/x$ for $x \geq n$, we have \[\delta(E) \geq F(|B_1| + |B_2| - |A|) = F(|E|).\]
    \end{proof}

\section{The fixed points on type spaces property (FPT)}
	
    The following is folklore:
    \begin{lem}
        Let $M$ be an $\mc{L}$-structure, and let $G = \Aut(M)$ with the pointwise convergence topology. Then, for each $n \geq 1$, $G$ acts continuously on the Stone space $S_n(M)$ of $n$-types with parameters in $M$, with the action given by \[g \cdot p(\bar{x}) = \{\phi(\bar{x}, g\bar{m}) : \phi(\bar{x}, \bar{m}) \in p(\bar{x})\}.\] That is, $G \curvearrowright S_n(M)$ with the action defined above is a $G$-flow.
    \end{lem}

    See \cite[Lemma 4.1]{MS23}, for a proof. (The proof is relatively straightforward and follows via a compactness argument.)
 
    Note that we define the action of $G$ on $\mc{L}(M)$-formulae as \[g \cdot \phi(\bar{x}, \bar{m}) = \phi(\bar{x}, g\bar{m}).\]		
	
    \begin{defn}[{\cite[Definition 4.2]{MS23}}]
        Let $M$ be an $\mc{L}$-structure and let $G = \Aut(M)$. We say that $M$ has the \emph{fixed points on type spaces property} (FPT) if every subflow of $G \curvearrowright S_n(M)$, $n \geq 1$, has a fixed point.
    \end{defn}
	
    Note that FPT is equivalent to every orbit closure $\ov{G \cdot p(\bar{x})}$ in $S_n(M)$ having a fixed point.

    The below lemma will play a key role in the proof of Theorem \ref{FPTMF}. (See Section \ref{top dyn intro} for the definition of a factor.)
    \begin{lem} \label{FPT passes to factors}
        Let $M$ be an $\mc{L}$-structure and let $G = \Aut(M)$. Suppose that $M$ has FPT. Then every subflow of each factor of $G \curvearrowright S_n(M)$, $n \geq 1$, has a fixed point.
    \end{lem}
    The proof is straightforward.
 
\section{An \texorpdfstring{$\omega$}{omega}-categorical structure such that no \texorpdfstring{$\omega$}{omega}-categorical expansion has FPT} \label{resultsec}

    We will now discuss the main result of this paper, which is new.
    
    \begin{thm} \label{FPTMF}
        There is an $\omega$-categorical structure $M$ such that no $\omega$-categorical expansion has FPT, the fixed points on type spaces property.
    \end{thm}
	
    The $\omega$-categorical structure $M$ in the above theorem will be a particular case of the $2$-sparse graph $M_F$, the $\omega$-categorical Hrushovski construction from Section \ref{omegacatsparsesec}.
	
    The proof will depend on the following key result from \cite{EHN19}:
    
    \begin{prop}[{\cite[Theorem 3.7]{EHN19}}] \label{MFfixorinforbs}
        Let $M$ be an infinite $2$-sparse graph in which all vertices have infinite degree. Let $G = \Aut(M)$, and let $H \leq G$.
		
        Consider the $G$-flow $G \curvearrowright \Or(M)$. If $H$ fixes a $2$-orientation of $M$, then $H$ has infinitely many orbits on $M^2$.
    \end{prop}
	
    Before giving the details of the proof of Theorem \ref{FPTMF}, we first give an informal general outline. 

    The informal overview of the proof is as follows. Let $G = \Aut(M)$. For each orientation $\tau \in \Or(M)$, we will define a notion of when a $1$-type $p(x) \in S_1(M)$ \emph{encodes} $\tau$ (see Figure \ref{N_taupic}), and we will have an associated ``decoding" $G$-morphism $u : S_1(M) \to 2^{M^2}$ sending each orientation-encoding $1$-type back to the orientation it encodes. We will show that each orientation has a $1$-type encoding it, and thus $u$ contains $\Or(M)$ in its image. Now, let $M'$ be an expansion of $M$ with FPT, and let $H$ denote its automorphism group. There is an $H$-factor map $w: S_1(M') \to S_1(M)$ given by the restriction map, and so composing with $u$ we see that $\Or(M)$ is an $H$-subflow of a factor of $S_1(M')$. Thus, by Lemma \ref{FPT passes to factors}, as $M'$ has FPT, $H$ fixes an orientation of $M$. Therefore, by Proposition \ref{MFfixorinforbs}, $H$ will have infinitely many orbits on $M^2$, and thus by the Ryll-Nardzewski theorem $M'$ cannot be $\omega$-categorical. 

    \medskip
  
    \textbf{We now start the formal details of the proof of Theorem \ref{FPTMF}, which proceeds in three parts.}

    \subsection*{Part 1: specify the control function.}
        
    We begin with a description of the control function $F$ and properties of the class $\mc{C}_F$ used to produce the structure $M$ for Theorem \ref{FPTMF}. It will become clearer in later steps why we take control functions satisfying the below conditions.
 
    \begin{lem} \label{specificF}
        Let $F$ be a control function for the class $\mc{C}_F$ satisfying the conditions of Definition \ref{CFdef}, and additionally assume the following conditions:
        \begin{itemize}
		\item $F$ is piecewise smooth, and its right derivative $F'(x)$ is decreasing;
		\item $F(1) = 2$, $F(2) = 3$; \item $F'(x) \leq \frac{2}{8x+1}$ for $x \geq 2$, where $F'$ denotes the right derivative. 
	\end{itemize}
		
        Then:
        \begin{enumerate}
            \item $\CF$ contains a point and an edge, and points and edges are $d$-closed in elements of $\CF$;
            \item $(\mc{C}_F, \leq_d)$ is a free amalgamation class;
		\item each vertex of $M_F$ has infinite degree (where $M_F$ is the \Fr limit of $(\mc{C}_F, \leq_d)$);
            \item if $a_0 a_1 \cdots a_{n-1}$ is a path, then $a_0 a_1 \cdots a_{n-1} \in \mc{C}_F$;
            \item $F(4) < 4$, $F(5) < 4, F(6) < 4$;
            \item if $abcd$ is a $4$-cycle, then $abcd \in \mc{C}_F$.
        \end{enumerate}
    \end{lem}
    \begin{proof} \hfill
	\begin{enumerate}
            \item As $F(1) = 2$, if $x$ is a point then $\delta(\{x\}) = 2 = F(|\{x\}|)$, so $\{x\} \in \CF$. If $ab$ is an edge, then $\delta(ab) = 3 = F(2)$, so $ab \in \CF$. As $F$ is strictly increasing, points and edges are $d$-closed in elements of $\CF$.
            \item Take $A, B_1, B_2 \in \mc{C}_F$ with $A \leq_d B_1, B_2$. Then as $F'(x) \leq \frac{2}{8x+1} < 1/x$ for $x \geq 2$, by Lemma \ref{CFfreeamalg} we need only check the case $|B_1|, |B_2| \leq 1$, and the only non-trivial case is where $A = \varnothing$: if $b_1, b_2$ are non-adjacent points then $\delta(\{b_1, b_2\}) = 4 > F(2)$. So $(\mc{C}_F, \leq_d)$ is a free amalgamation class.
            \item Let $k \geq 1$. Let $ax \in \CF$ be an edge. The point $a$ is $d$-closed in $ax$, and so by taking the free amalgamation of $k$ copies $ax_1, \cdots, ax_k$ of $ax$ over $a$, we have that the star graph $S_k$ is in $\CF$ (where $S_k$ is the complete bipartite graph $K_{1, k}$). Using the $\leq_d$-extension property of $M_F$, this implies that each vertex of $M_F$ has infinite degree.
            \item Proceed by induction, and obtain $a_0 \cdots a_{n-1} \in \mc{C}_F$ by the free amalgamation of $a_0 \cdots a_{n-2}$ and $a_{n-2} a_{n-1}$ over $a_{n-2}$.
            \item $F$ is strictly increasing, and so it suffices to show $F(6) < 4$. $F(6) \leq F(2) + \int_2^6 \frac{2}{8x + 1} \d x = 3 + \frac{1}{4} \log(49) - \frac{1}{4} \log(17) < 4$.
            \item Let $abcd \sub M_F$ be a $4$-cycle. Then $\delta(abcd) = 4 > F(4)$. For $C \subsetneq abcd$, $C$ either consists of a path of length 2, an edge, two non-adjacent points or a single point. All of these lie in $\mc{C}_F$.
        \end{enumerate}
    \end{proof}
	
    Throughout the rest of the proof of Theorem \ref{FPTMF} in Section \ref{resultsec}, we will assume $F$ is a control function satisfying the conditions of Lemma \ref{specificF}, and we write $M = M_F$. The first three conditions of the above lemma are relatively standard; the fourth condition is the one that is particularly specific to our example, constituting a mild additional restriction on $F$.
	
    Note that control functions satisfying the conditions of the above lemma do indeed exist: for example, take $F$ piecewise linear with $F(0) = 0$, $F(1) = 2$, $F(2) = 3$, and then for $x \geq 2$ define $F(x) = \frac{1}{4} \log(8x + 1) + 3 - \frac{1}{4} \log(17)$.

    \subsection*{Part 2: types encoding orientations, the encoding lemma, and its use in proving the main theorem.}

    Given an orientation $\tau \in \Or(M)$, we will define a particular notion of when a $1$-type \emph{encodes} $\tau$.

    We write $\mc{L}$ for the language of graphs (so $M$ is an $\mc{L}$-structure).

    \begin{defn}
        For $a, b \in M$, we define the \emph{label formula} $f(x, a, b)$ in the language $\mc{L}(M)$ (with constants from $M$ and free variable $x$) to be: 
	\begin{align*}
            &f(x, a, b) \equiv\, (x \neq a \wedge x \neq b \wedge a \neq b \wedge a \sim b) \, \wedge \\
            &(\ex l_1, l_2, l_3, l_4) ((\bigwedge_{i < j} l_i \neq l_j) \wedge (\bigwedge_i l_i \neq x \wedge l_i \neq a \wedge l_i \neq b) \,\wedge \\
            &(x \sim l_1 \wedge l_1 \sim l_2 \wedge l_2 \sim l_3 \wedge l_3 \sim l_4 \wedge l_4 \sim l_1 \wedge l_2 \sim a \wedge l_4 \sim a \wedge l_3 \sim b)). 
        \end{align*} 
        (See Figure \ref{N_taupic}.)
        
        \medskip
        
        Let $\tau \in \Or(M)$, and let $p(x) \in S_1(M)$. We say that $p(x)$ \emph{encodes} $\tau$ if $p(x)$ contains the following set of formulae:
        \[\{f(x, a, b) : (a, b) \in \tau\} \cup \{\neg f(x, a, b) : (a, b) \in M^2 \setminus \tau\}.\]

        Informally, $p(x)$ encodes $\tau$ if, for every pair $(a, b) \in M^2$, we have that $(a, b) \in \tau$ iff $a, b$ are adjacent and $(a, b)$ has a ``label structure" $L_{(a, b)} = \{x, l_1^{(a, b)}, l_2^{(a, b)}, l_3^{(a, b)}, l_4^{(a, b)}\}$ attached to it, where all label structures intersect exactly in the ``head vertex" $x$. See Figure \ref{N_taupic} for an example of this (where $p(x)$ has been realised as a point $c$).
	
        We define the \emph{decoding map} $u : S_1(M) \to 2^{{M}^2}$ by \[u(p(x)) = \{(a, b) \in {M}^2 : f(x, a, b) \in p(x)\}.\] (Note that we will often use subset notation when formally we in fact mean the characteristic function of that subset within ${M}^2$.)

        It is immediate that if $p(x)$ encodes $\tau$, then $u(p(x)) = \tau$.
    \end{defn}

    \begin{figure} 
        \centering
        \def\svgscale{0.6}
        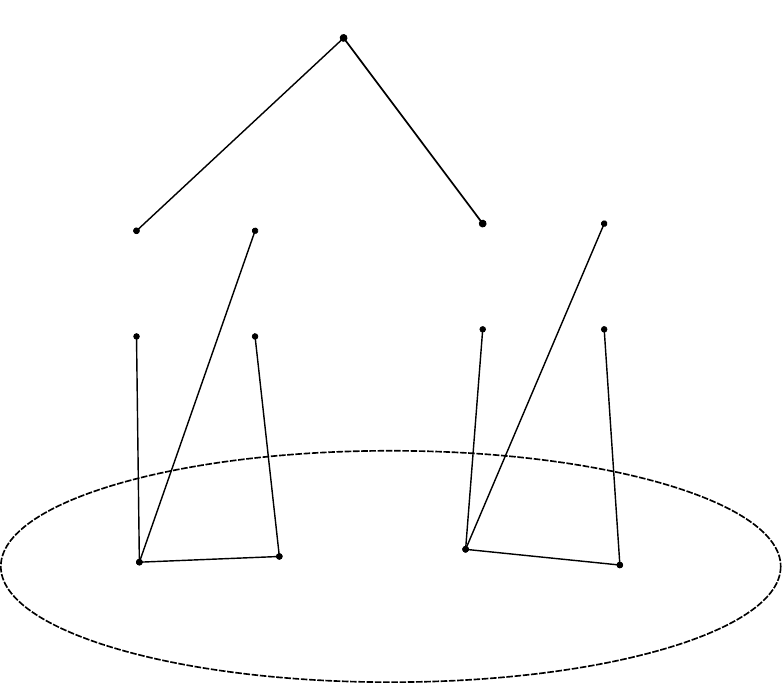
        \caption{} \label{N_taupic}
    \end{figure}

    The proof of the following is straightforward.
    \begin{lem}
	The map $u$ is a $G$-flow morphism.
    \end{lem}

    Now we turn to the key result used in the proof of Theorem \ref{FPTMF}, which we call the \emph{encoding lemma}:

    \begin{lem} \label{can encode ors}
        For each orientation $\tau \in \Or(M)$, there exists a type $p(x) \in S_1(M)$ encoding $\tau$. Thus $\Or(M)$ is a subflow of the image of the decoding map $u$.
    \end{lem}

    Before proving Lemma \ref{can encode ors}, whose proof involves a significant amount of technical work, we show how to use it to prove Theorem \ref{FPTMF}.

    \begin{proof}[Proof of Theorem \ref{FPTMF} given Lemma \ref{can encode ors}]
        Let $M'$ be an expansion of $M$ with FPT. Let $H = \Aut(M)$.
		
        We have a surjective $H$-flow morphism $w : S_1(M') \to S_1(M)$ given by restriction, i.e.\ \[w(p(x)) = \{\varphi(x) \in p(x) : \varphi(x)\text{ is a formula in the language }\mc{L}(M)\}.\]
		
        We have that $u : S_1(M) \to 2^{{M}^2}$ is a $G$-flow morphism, and $\Or(M)$ is a $G$-subflow of $2^{{M}^2}$ contained in the image of $u$. By considering $u$ as an $H$-flow morphism to its image, we have that $u \circ w$ is an $H$-factor of $S_1(M')$ with $\Or(M)$ as a subflow of its image. As $M'$ has FPT, by Lemma \ref{FPT passes to factors} we see that $H$ fixes an orientation on $M$. By Theorem \ref{MFfixorinforbs}, $H$ has infinitely many orbits on $M^2$, and so is not oligomorphic. Therefore $M'$ is not $\omega$-categorical, by the Ryll-Nardzewski theorem.
    \end{proof}

    \subsection*{Part 3: the proof of the encoding lemma.}

    We now prove Lemma \ref{can encode ors}. This forms the bulk of the technical work in this paper.

    Let $\tau \in \Or(M)$. To show that there exists a type encoding $\tau$, it suffices to show that the set of formulae \[\Lambda(x) = \{f(x, a, b) : (a, b) \in \tau\} \cup \{\neg f(x, a, b) : (a, b) \in M^2 \setminus \tau\}\] is finitely satisfiable in $M$ itself: this implies via compactness that there exists a type $p(x) \in S_1(M)$ containing this set of formulae.

    Again, before beginning the proof of Lemma \ref{can encode ors} we provide a brief informal overview. In order to show the finite satisfiability of $\Lambda(x)$, we will take a finite $d$-closed substructure $A$ of $M$ and show that the set $\Lambda_A(x)$ is satisfiable in $M$, where $\Lambda_A(x)$ consists of the formulae in $\Lambda(x)$ with parameters only from $A$. We will construct a graph $B$ with head vertex $c$ of label structures over $A$ (as in Figure \ref{N_taupic}) such that $A \leq_d B$ and $B \in \CF$. Therefore in fact we may assume $A \leq_d B \leq_d M$, using the $\leq_d$-extension property. We will then show that the only label structures in $M$ over any pair of elements of $A$ can be found in $B$, using the fact that $B$ is $d$-closed in $M$, and thus we have that $M \models \Lambda_A(c)$. 

    We now begin the formal proof. Let $A \leq_d M$ be finite. We define a graph $B$ as follows.
    
    \begin{itemize}
        \item $B$ includes $A$ as a substructure;
        \item add a new vertex $c$ to $B$, with $c \notin A$;
         \item for $(a, b) \in \tau|_A$ (i.e.\ the edge $ab$ is oriented from $a$ to $b$ in the orientation $\tau$ and $ab$ is an edge of $A$), add to $B$ four new vertices $l_1^{(a, b)}, l_2^{(a, b)}, l_3^{(a, b)}, l_4^{(a, b)}$ and new edges \[cl_1^{(a, b)},  l_1^{(a, b)} l_2^{(a, b)}, l_2^{(a, b)} l_3^{(a, b)}, l_3^{(a, b)} l_4^{(a, b)}, l_4^{(a, b)} l_1^{(a, b)}\] and add two edges $l_2^{(a, b)}a, l_4^{(a, b)}a$ (to the ``start vertex" $a$) and one edge $l_3^{(a, b)}b$ (to the ``end vertex" $b$).
    \end{itemize}
 
    For $(a, b) \in \tau|_A$, let $L_{(a, b)} = \{c, l_1^{(a, b)}, l_2^{(a, b)}, l_3^{(a, b)}, l_4^{(a, b)}\}$. Informally, each $(a, b) \in \tau|_A$ has its orientation labelled by $L_{(a, b)}$. 
    
    We have that \[B = A \,\cup \bigcup_{(a, b) \in \tau|_A} L_{(a, b)}\] and the $L_{(a, b)}$ intersect only in $c$. We will show that $A \leq_d B$ and $B \in \CF$.
    
    It is recommended that the reader consult Figure \ref{N_taupic} during the technical lemmas in this part of the proof.
 
    \begin{lem} \label{A_istrongB_i}
        We have $A \leq_d B$.
    \end{lem}
    \begin{proof}
        For $A \subsetneq C \sub B$, we need to show $\delta(C) > \delta(A)$.
		
        First consider the case where $A$ consists of a single edge $ab$, with $(a, b) \in \tau$ (recall that we chose the control function $F$ so that edges are always $d$-closed). Then, suppressing subscripts for notational convenience, we have $B = \{c, l_1, l_2, l_3, l_4, a, b\}$. We calculate the relative predimension of some $A \subsetneq C \sub B$ in the table below.
		
	\begin{center} \small
		\begin{tabular} {|c|c|}
			\hline
			$C \setminus A$ & $\delta(C/A)$ \\
			\hline
			$l_2$ & 1 \\
			$l_3$ & 1 \\
			$l_4$ & 1 \\
			$l_1, l_2$ & 2 \\
			$l_1, l_4$ & 2 \\
			$l_2, l_3$ & 1 \\
			$l_3, l_4$ & 1 \\
			$c, l_1$ & 3 \\
			$l_1, l_2, l_3$ & 2 \\
			$l_1, l_2, l_4$ & 2 \\
			$l_1, l_3, l_4$ & 2 \\
			$l_2, l_3, l_4$ & 1 \\
			$l_1, l_2, l_3, l_4$ & 1 \\
			$c, l_1, l_2, l_3, l_4$ & 2 \\
			\hline
		\end{tabular}
	\end{center}
		
        The remaining cases result from free amalgamations over $A$, and so also have positive predimension (as if two graphs $X, Y$ are freely amalgamated over $Z$, then $\delta((X \cup Y) / Z) = \delta(X/Z) + \delta(Y/Z)$). The remaining cases are where $C \setminus A$ is equal to $\{l_1\}, \{c\}, \{l_1, l_3\}, \{l_2, l_4\}$ or $\{c\} \cup X$, where $X \sub \{l_2, l_3, l_4\}$.
		
        Now consider the general case of finite $A \leq_d M$. Given $A \subsetneq C \sub B$, the vertices of $C$ consist of $A$ together with subsets $J_{(a, b)}$ of $L_{(a, b)}$ for each $(a, b) \in \tau|_A$. For $(a, b) \in \tau|_A$, let $J'_{(a, b)} = J^{}_{(a, b)} \cup A$. 
		
        If $c \notin C$, then the $J'_{(a, b)}$ are freely amalgamated over $A$, and so from the single-edge case we see that $\delta(C / A) > 0$. 
		
        We now consider the case where $c \in C$. If $l_{1}^{(a, b)} \notin J^{\alignsub}_{(a, b)}$ for all $(a, b) \in \tau|_A$, then $C$ consists of a vertex $c$ with no neighbours together with a free amalgamation over $A$ of each of the $J'_{(a, b)} \setminus \{c\}$, for $(a, b) \in \tau|_A$. So, from the single-edge case and the fact that $\delta(\{c\}) = 2$, we have that $\delta(C / A) > 0$.
		
        If $c \in C$ and there exists $(a', b') \in \tau|_A$ with $l_{1}^{(a', b')} \in J^{}_{(a', b')}$, then $C$ is the free amalgam over $A$ of each of the $J'_{(a, b)} \setminus \{c\}$ for which $l_{1}^{(a, b)} \notin J^{\alignsub}_{(a, b)}$ (where $(a, b) \in \tau|_A$), together with \[\bigcup \{J'_{(a, b)} : (a, b) \in \tau|_A,\, l_{1}^{(a, b)} \in J^{}_{(a, b)}\}.\] Therefore we need only consider the case where $l_{1}^{(a, b)} \in J^{}_{(a, b)}$ for all $(a, b) \in \tau|_A$. The single-edge calculation shows that $\delta(J_{(a, b)} \setminus \{c\} / A) \geq 1$ for each $J_{(a, b)}$, and these $J'_{(a, b)} \setminus \{c\}$ are freely amalgamated over $A$. Each addition of an edge $cl_1^{(a, b)}$ reduces the predimension by one, but the single addition of the vertex $c$ adds two to the predimension, so in total $\delta(C/A) > 0$.
    \end{proof}
	
    \begin{lem} \label{smallstrsCF}
        For $(a, b) \in \tau|_A$, the substructures of $B$ given by $\{a, b, l_1^{(a, b)}, l_2^{(a, b)}, l_3^{(a, b)}, l_4^{(a, b)}\}$ and $L_{(a, b)}$ lie in $\CF$.
    \end{lem}
    \begin{proof}
        We write $l_1, l_2, l_3, l_4$, suppressing superscripts.
		
        To show that $\{a, b, l_1, l_2, l_3, l_4\} \in \mc{C}_F$, we consider each subset $C \sub \{a, b, l_1, l_2, l_3, l_4\}$ and show that $\delta(C) \geq F(|C|)$. To speed up the process of checking each subset $C$, in the below table we show that certain subsets $C \sub \{a, b, l_1, l_2, l_3, l_4\}$ lie in $\mc{C}_F$, and therefore every $C' \sub C$ must satisfy $\delta(C') \geq F(|C'|)$.
	\begin{center} \small
		\begin{tabular} {|c|c|}
			\hline
			$C$ & Proof that $C \in \mc{C}_F$ \\
			\hline
			$l_1 l_2 l_3 l_4, l_2 l_3 ab, l_3 l_4 ab, l_1 l_2 l_4 a, l_2 l_3 l_4 a$ & $C$ is a $4$-cycle\\
			$l_1 l_2 l_3 ab$ & free amalgam of $l_2 l_3 ab$, $l_1 l_2$ over $l_2$ \\
			$l_1 l_3 l_4 ab$ & free amalgam of $l_3 l_4 ab$, $l_1 l_4$ over $l_4$ \\
			$l_1 l_2 l_4 ab$ & free amalgam of $l_1 l_2 l_4 a$, $ab$ over $a$ \\
			$l_1 l_2 l_3 l_4 b$ & free amalgam of $l_1 l_2 l_3 l_4$, $l_3 b$ over $l_3$ \\
			\hline
		\end{tabular}
	\end{center}
        We now check the remaining subsets $C \sub \{a, b, l_1, l_2, l_3, l_4\}$ by directly calculating the predimension:
	\begin{center} \small
		\begin{tabular} {|c|c|c|}
			\hline
			$C$ & $\delta(C)$ & $F(|C|)$ \\
			\hline
			$l_2 l_3 l_4 ab$ & $4$ & $F(5) < 4$ \\
			$l_1 l_2 l_3 l_4 a$ & $4$ & $F(5) < 4$ \\
			$l_1 l_2 l_3 l_4 ab$ & $4$ & $F(6) < 4$ \\
			\hline
		\end{tabular}
	\end{center}
		
        We have now shown that $\{a, b, l_1, l_2, l_3, l_4\} \in \mc{C}_F$. For the second part of the lemma, we obtain $L_{(a, b)} \in \mc{C}_F$ via the free amalgam of $L_{(a, b)}$ and $cl_1$ over $l_1$ (recalling that we have defined our control function $F$ so that points are always $d$-closed). 
    \end{proof}
    \begin{lem} \label{B in CF}
	We have that $B \in \CF$.
    \end{lem}
    \begin{proof}
        We have to show that $\delta(C) \geq F(|C|)$ for $C \sub B$. The vertices of $C$ consist of $C \cap A$ together with subsets $J_{(a, b)}$ of $L_{(a, b)}$ for each $(a, b) \in \tau|_A$ (some of these $J_{(a, b)}$ may be empty). For $(a, b) \in \tau|_A$, let $J'_{(a, b)} = J^{}_{(a, b)} \cup (C \cap A)$.
		
        First we consider the case where $c \notin C$. $C$ is then the free amalgam of the $J'_{(a, b)}$ (where $(a, b) \in \tau|_A$) over $C \cap A$. Given that $\mc{C}_F$ is a free amalgamation class and $C \cap A \leq_d C$, it therefore suffices to show that $J'_{(a, b)} \in \mc{C}_F$ for $(a, b) \in \tau|_A$. Fix $(a, b) \in \tau|_A$. To show that $J'_{(a, b)} \in \mc{C}_F$, as $J'_{(a, b)}$ is a free amalgam of $J_{(a, b)} \cup (\{a, b\} \cap C)$ and $C \cap A \in \mc{C}_F$ over $\{a, b\} \cap C \in \mc{C}_F$, it suffices to show that $J_{(a, b)} \cup (\{a, b\} \cap C)$ lies in $\mc{C}_F$, and we have already checked this in Lemma \ref{smallstrsCF}.
		
        Now we consider the case where $c \in C$. If $l_{1}^{(a, b)} \notin J^{}_{(a, b)}$ for each $(a, b) \in \tau|_A$, then $C$ consists of a vertex $c$ with no neighbours together with the free amalgam over $C \cap A$ of each $J'_{(a, b)} \setminus \{c\}$, and so we are done by the first case in the previous paragraph. Otherwise, $C$ is the free amalgam over $C \cap A$ of \[\bigcup \{J'_{(a, b)} : l_1^{(a, b)} \in J^{}_{(a, b)}, (a, b) \in \tau|_A\}\] with each $J'_{(a, b)} \setminus \{c\}$ for which $l_{1}^{(a, b)} \notin J^{}_{(a, b)}$, and so using the first case considered above we may reduce to the case where each non-empty $J_{(a, b)}$ contains $l_1^{(a, b)}$.
		
        Similarly, we may exclude the case where $C$ contains sets $J_{(a, b)}$ with $J_{(a, b)} = \{c, l_1^{(a, b)}, l_3^{(a, b)}\}$, as $C$ is the free amalgam over $C \cap A$ of 
		
        \begin{align*}
            &\bigcup \left\{J'_{(a, b)} : (a, b) \in \tau|_A,\, J_{(a, b)} \neq \{c, l_1^{(a, b)}, l_3^{(a, b)}\}\right\} \;\cup \\
            &\bigcup \left\{\{c, l_1^{(a, b)}\} \cup (C \cap A) : J_{(a, b)} = \{c, l_1^{(a, b)}, l_3^{(a, b)}\}\right\}
        \end{align*}
		
        with each $\{l_3^{(a, b)}\} \cup (C \cap A)$ (which lies in $\mc{C}_F$, using Lemma \ref{smallstrsCF}) for which $J_{(a, b)} = \{c, l_1^{(a, b)}, l_3^{(a, b)}\}$. We may likewise freely amalgamate over $c$ to exclude the cases where $C$ contains sets $J_{(a, b)}$ for which $J_{(a, b)} = \{c, l_1^{(a, b)}\}$, or for which $J_{(a, b)}$ is any subset of $L_{(a, b)}$ but we have $a, b \notin C \cap A$.
		
        So, the case remaining is where $C$ consists of $C \cap A$ together with sets $J_{(a, b)}$ containing $c, l_1^{(a, b)}$ and at least one of $l_2^{(a, b)}, l_4^{(a, b)}$, where each $J_{(a, b)}$ has some edge to $C \cap A$. We need to show that $\delta(C) \geq F(|C|)$. 
		
        We now calculate the relative predimension over $A \cup \{c\}$ of each remaining possible $J_{(a, b)} \cup X$, $X \sub \{a, b\}$, in the following table, where we label each structure as $Y_i$, $1 \leq i \leq 11$:
		
	\begin{center} \small
		\begin{tabular} {|c|c|c|}
			\hline
			$J_{(a, b)} \cup X$ & Label & $\delta(J_{(a, b)} \cup X / A \cup \{c\})$ \\
			\hline
			$c l_1 l_2 a$ & $Y_1$ & 1 \\
			$c l_1 l_4 a$ & $Y_2$ & 1 \\
			$c l_1 l_2 l_3 a$ & $Y_3$ & 2 \\
			$c l_1 l_2 l_3 b$ & $Y_4$ & 2 \\
			$c l_1 l_2 l_3 ab$ & $Y_5$ & 1 \\
			$c l_1 l_3 l_4 a$ & $Y_6$ & 2 \\
			$c l_1 l_3 l_4 b$ & $Y_7$ & 2 \\
			$c l_1 l_3 l_4 ab$ & $Y_8$ & 1 \\
			$c l_1 l_2 l_3 l_4 a$ & $Y_9$ & 1 \\
			$c l_1 l_2 l_3 l_4 b$ & $Y_{10}$ & 2 \\
			$c l_1 l_2 l_3 l_4 ab$ & $Y_{11}$ & 0 \\
			\hline
		\end{tabular}
	\end{center}
		
        We write $k_i$ for how many times $Y_i$ occurs in $C$. We also write $\delta_i = \delta(Y_i / A \cup \{c\})$. Let $\lambda_i = |\{l_1, l_2, l_3, l_4\} \cap Y_i|$.
		
        Then, recalling that the vertex $c$ also adds 2 to the predimension, we have that \[ \delta(C) = \sum_{1 \leq i \leq 11} \delta_i k_i + 2 + \delta(C \cap A). \] Now, 
        \begin{align*}
		F(|C|) &= F(1 + |C \cap A| + \sum_{1 \leq i \leq 11} \lambda_i k_i) \\ &\leq F(1 + |C \cap A| + 4\sum_{1 \leq i \leq 11} k_i)\\ &=
		F(1 + |C \cap A| + 4(k_4 + k_7 + k_{10}) + 4\sum_{i \leq 11,\, i \notin \{4,7,10\}} k_i).
        \end{align*}
		
        As $\tau|_{C \cap A}$ is a 2-orientation, we have that each $a \in C \cap A$ can have at most two label structures with $a$ as the starting vertex (i.e.\ with edges to $a$ from $l_2, l_4$), and so \[\sum_{i \leq 11,\, i \notin \{4,7,10\}} k_i \leq 2|C \cap A|.\] So $F(|C|) \leq F(8|C \cap A| + 4(k_4 + k_7 + k_{10}) + 1 + |C \cap A|)$.
		
        As $F(u + v) \leq F(u) + vF'(u)$ and $F'(x) \leq \frac{2}{8x + 1}$ for $x \geq 2$, we have that if $|C \cap A| \geq 2$, then 
        \begin{align*}
		F(|C|) &\leq F(|C \cap A|) + \frac{2}{8|C \cap A| + 1} (8|C \cap A| + 4(k_4 + k_7 + k_{10}) + 1) \\ &< F(|C \cap A|) + 2 + k_4 + k_7 + k_{10} \\ &\leq \delta(C).
        \end{align*} 
		
        If $|C \cap A| = 1$, then 
	\begin{align*}
		F(|C|) &\leq F(1 + |C \cap A|) + (8|C \cap A| + 4(k_4 + k_7 + k_{10}))F'(1 + |C \cap A|) \\ &= 3 + \frac{2}{8 \cdot 2 + 1} (8|C \cap A| + 4(k_4 + k_7 + k_{10})) \\ &< 4 + \frac{8}{17}(k_4 + k_7 + k_{10}) \\ &\leq \delta(C)
	\end{align*}
	(as $\delta(C \cap A) = 2$).
    \end{proof}
        
    \begin{lem} \label{fortauthereexistsptau}
        Let $\tau \in \Or(M)$. Then the set of formulae \[\Lambda(x) = \{f(x, a, b) : (a, b) \in \tau\} \cup \{\neg f(x, a, b) : (a, b) \in M^2 \setminus \tau\}\] is finitely satisfiable in $M$.
    \end{lem}
    \begin{proof}
        Let $A \leq_d M$ be finite, and let
        \[
            \Phi_A(x) = \{f(x, a, b) : (a, b) \in \tau|_A\},
            \Psi_A(x) = \{\neg f(x, a, b) : (a, b) \in A^2 \setminus \tau\}.
        \]
        Let $\Lambda_A(x) = \Phi_A(x) \cup \Psi_A(x)$. We will show that $\Lambda_A(x)$ is satisfiable in $M$. 
        
        Let $B \supseteq A$ be as constructed previously, with distinguished head vertex $c$. As $A \leq_d B$ (Lemma \ref{A_istrongB_i}) and $B \in \CF$ (Lemma \ref{B in CF}), we may use the $\leq_d$-extension property of $M$ to assume that $A \leq_d B \leq_d M$.

        It is immediate from the construction of $B$ that $B \models \Phi_A(c)$ and hence $M \models \Phi_A(c)$, as for each $(a, b) \in \tau|_A$, there is a label structure $L_{(a, b)}$ attached.

        We now show that $M \models \Psi_A(c)$. It suffices to show that for $(a, b) \in A^2$, if $M \models f(c, a, b)$ then the $l_i$, $1 \leq i \leq 4$, that $f(c, a, b)$ specifies must lie in $\cl^d_{M}(\{a, b, c\})$ and therefore in $B$, as $B \leq_d M$. We show that $\{l_1, l_2, l_3, l_4\} \sub \cl^d_M(\{a, b, c\})$ in the table below.
		
	\begin{center} \small
		\begin{tabular} {|c|c|}
			\hline
			$X/Y$ & $\delta(X/Y)$ \\
			\hline
			$l_1, l_2, l_3, l_4 / a, b, c$ & 0 \\
			$l_1, l_2, l_3 / l_4, a, b, c$ & $-1$ \\
			$l_1, l_2, l_4 / l_3, a, b, c$ & $-1$ \\
			$l_1, l_3, l_4 / l_2, a, b, c$ & $-1$ \\
			$l_2, l_3, l_4 / l_1, a, b, c$ & $-1$ \\
			$l_1, l_2 / l_3, l_4, a, b, c$ & $-1$ \\
			$l_1, l_3 / l_2, l_4, a, b, c$ & $-2$ \\
			$l_1, l_4 / l_2, l_3, a, b, c$ & $-1$ \\
			$l_2, l_3 / l_1, l_4, a, b, c$ & $-1$ \\
			$l_2, l_4 / l_1, l_3, a, b, c$ & $-2$ \\
			$l_3, l_4 / l_1, l_2, a, b, c$ & $-1$ \\
			$l_1 / l_2, l_3, l_4, a, b, c$ & $-1$ \\
			$l_2 / l_1, l_3, l_4, a, b, c$ & $-1$ \\
			$l_3 / l_1, l_2, l_4, a, b, c$ & $-1$ \\
			$l_4 / l_1, l_2, l_3, a, b, c$ & $-1$ \\
			\hline
		\end{tabular}
	\end{center}
	This completes the proof of Lemma \ref{fortauthereexistsptau}.
    \end{proof}

    The above lemma implies, via compactness, that there exists a type $p(x)$ containing the set $\{f(x, a, b) : (a, b) \in \tau\} \cup \{\neg f(x, a, b) : (a, b) \in M^2 \setminus \tau\}$, and thus $p(x)$ encodes $\tau$. This completes the proof of the encoding lemma (Lemma \ref{can encode ors}), and therefore the proof of Theorem \ref{FPTMF}.

\bibliographystyle{abbrv}
\bibliography{super}

\end{document}